\documentclass[12pt]{amsart}

\font\bbbld=msbm10 scaled\magstephalf

\newcommand{\bM}{\bar{M}}

\newcommand{\bfR}{\hbox{\bbbld R}}

\newcommand{\ol}{\overline}
\newcommand{\ul}{\underline}

\newtheorem{theorem}{Theorem}[section]

 \theoremstyle{definition}

\theoremstyle{remark}

\numberwithin{equation}{section}

\begin{document}

\title[Hessian equations of parabolic type]
{Second order estimates for Hessian equations of parabolic type
on Riemannian manifolds}


\author{Heming Jiao}
\address{Department of Mathematics, Harbin Institute of Technology,
         Harbin, 150001, China}
\email{jiao@hit.edu.cn}


\begin{abstract}
In this paper, we establish the second order
estimates of solutions to the first initial-boundary value problem
for general Hessian type fully nonlinear parabolic equations on Riemannian manifolds.
The techniques used in this article can work for a wide range of fully
nonlinear PDEs under very general conditions.

\end{abstract}
\maketitle


\section{Introduction}

Let $(M^n, g)$ be a compact Riemannian manifold of
dimension $n \geq 2$ with smooth boundary $\partial M$
and $\bM := M \cup \partial M$.
We will study the equation
\begin{equation}
\label{eqn}
f (\lambda(\nabla^{2}u + A [u])) - u_t = \psi (x, t, u, \nabla u)
\end{equation}
in $M_T = M \times (0,T] \subset M \times \mathbb{R}$, where $f$ is
a symmetric smooth function of $n$ variables, $\nabla^2 u$ denotes the
Hessian of $u (x, t)$ with respect to $x \in M$,
$A [u] = A (x, t, \nabla u)$ is a $(0,2)$ tensor on
$\bM$ which may depend on $t \in [0, T]$ and $\nabla u$
and
\[ \lambda (\nabla^2 u + A [u]) = (\lambda_1 ,\ldots,\lambda_n) \]
denotes the
eigenvalues of $\nabla^2 u + A [u]$ with respect to the metric $g$.

In this paper we are mainly concerned with the \emph{a priori} $C^2$ estimates for solutions
to \eqref{eqn} with boundary condition
\begin{equation}
\label{eqn-bd}
u = \varphi \mbox{ on }\mathcal{P} M_T,
\end{equation}
where
$\varphi \in C^{\infty}(\overline{\mathcal{P}M_T})$ satisfying
$\lambda (\nabla^{2} \varphi(x,0) + A[\varphi(x,0)]) \in \Gamma$ for all $x \in \bM$.
Here $\mathcal{P} M_T
= B M_T \cup S M_T$ is the parabolic boundary of $M_T$ with $B M_T = M
\times \{0\}$ and $S M_T = \partial M \times [0,T]$.

The idea of this paper is mainly from Guan and Jiao \cite{GJ} where the authors studied
the second order estimates for the elliptic counterpart of \eqref{eqn}:
\begin{equation}
\label{3I-10}
f (\lambda (\nabla^2 u + A (x, u, \nabla u)))
  = \psi (x, u, \nabla u).
\end{equation}
Comparing with the elliptic case,
the main difficulty in deriving the second order estimates for the parabolic equation
\eqref{eqn} is from its degeneracy which is overcome by using the strict subsolution
in this paper. Surprisingly, thanks to the strict subsolution, we are able to relax
some restrictions to $f$.
Again because of the degeneracy, we do not get the higher estimates and the
existence of classical solution. It is useful to consider viscosity solutions to
\eqref{eqn} which will be addressed in forthcoming papers.

The first initial-boundary value problem for equation of form \eqref{eqn} in $\mathbb{R}^n$
with $A \equiv 0$ and $\psi = \psi (x,t)$ was studied by Ivochkina and Ladyzhenskaya in \cite{IL1}
(when $f = \sigma_n^{1/n}$) and \cite{IL2}. Jiao and Sui treated the case that
$A \equiv \chi (x)$ and $\psi = \psi (x,t)$ on Riemannian manifolds using the techniques
of \cite{G} and \cite{GJ}.
For the elliptic Hessian equations on manifolds, we refer
the readers to Li \cite{LiYY90}, Urbas \cite{U}, Guan \cite{G1,G,G14}, Guan and Jiao \cite{GJ} and
their references.

As in \cite{CNS}, in which the authors studied the equations \eqref{3I-10}
with $A \equiv 0$ and $\psi = \psi (x)$
in a bounded domain of $\mathbb{R}^n$, $f \in C^\infty (\Gamma) \cap C^0 (\overline{\Gamma})$
is assumed to be defined on $\Gamma$, where $\Gamma$ is an open, convex,
symmetric proper subcone of $\mathbb{R}^{n}$ with vertex at the origin and
\[\Gamma^{+} \equiv \{\lambda \in \mathbb{R}^{n}:\mbox{ each component }
   \lambda_{i} > 0\} \subseteq \Gamma,\]
and to satisfy the following structure conditions in this paper:
\begin{equation}
\label{f1}
f_{i} \equiv \frac{\partial f}{\partial \lambda_{i}} > 0 \mbox{ in } \Gamma,\ \ 1\leq i\leq n,
\end{equation}
\begin{equation}
\label{f2}
f\mbox{ is concave in }\Gamma,
\end{equation}
and
\begin{equation}
\label{f5}
f > 0 \mbox{ in } \Gamma, \ \ f = 0 \mbox{ on } \partial \Gamma.
\end{equation}

Typical examples are given by $f = \sigma^{1/k}_k$ and $f = (\sigma_k / \sigma_l)^{1/(k - l)}$,
$1 \leq l < k \leq n$, defined in the cone
$\Gamma_{k} = \{\lambda \in \mathbb{R}^{n}: \sigma_{j} (\lambda) > 0, j = 1, \ldots, k\}$,
where $\sigma_{k} (\lambda)$ are the elementary symmetric functions
\[\sigma_{k} (\lambda) = \sum_ {i_{1} < \ldots < i_{k}}
\lambda_{i_{1}} \ldots \lambda_{i_{k}},\ \ k = 1, \ldots, n.\]
Another interesting example is $f = \log P_k$, where
\[ P_k (\lambda) := \prod_{i_1 < \cdots < i_k}
(\lambda_{i_1} + \cdots + \lambda_{i_k}), \;\; 1 \leq k \leq n,\]
defined in the cone
\[ \mathcal{P}_k : = \{\lambda \in \bfR^n:
      \lambda_{i_1} + \cdots + \lambda_{i_k} > 0\}. \]

We call a function $u(x,t)$ admissible if $\lambda(\nabla^{2} u + A[u]) \in \Gamma$ in $M \times [0, T]$.
It is shown in \cite{CNS} that \eqref{f1} ensures that equation \eqref{eqn}
is parabolic for admissible solutions. \eqref{f2} means that the function $F$ defined
by $F (A) = f (\lambda [A])$ is concave for $A \in \mathcal{S}^{n \times n}$ with
$\lambda [A] \in \Gamma$, where $\mathcal{S}^{n \times n}$ is the set of $n \times n$
symmetric matrices.

Throughout the paper we assume $A [u]$ is smooth on $\bM_T$ for $u \in C^{\infty} (\bM_T)$,
$\psi \in C^{\infty} (T^*\bM \times [0,T] \times \bfR)$ (for convenience we shall write
$\psi = \psi (x, t, z, p)$ for $(x, p) \in T^*\bM$, $t \in [0, T]$ and $z \in \bfR$ though).
Note that for fixed $(x, t) \in \bM_T$ and $p \in T^*_x M$,
\[ A (x, t, p): T^*_x M \times T^*_x M \rightarrow \bfR \]
is a symmetric bilinear map. We shall use the notation
\[ A^{\xi \eta} (x, \cdot, \cdot) := A (x, \cdot, \cdot) (\xi, \eta), \;\;
\xi,  \eta \in T^*_x M \]
and, for a function $v \in C^{2,1}_{x,t} (M_T)$,  $A [v] := A (x, t, \nabla v)$,
$A^{\xi \eta} [v] := A^{\xi \eta} (x, t, \nabla v)$ (see \cite{GJ}).

In this paper
we assume that there exists an admissible function
$\underline{u} \in C^{2} (\bM_T)$ satisfying
\begin{equation}
\label{sub}
f(\lambda(\nabla^{2} \underline{u} + A [\ul u])) - \underline{u}_{t}
    \geq \psi(x, t, \ul u, \nabla \ul u) + \delta_0 \mbox{ in } M \times [0, T].
\end{equation}
for some positive constant $\delta_0$ with $\ul u = \varphi$ on $\partial M \times [0, T]$
and $\ul u \leq \varphi$ in $M \times \{0\}$.



We shall prove the following Theorem.
\begin{theorem}
\label{jsui-th1}
Let $u \in C^4 (\bM_T)$ be an
admissible solution of (\ref{eqn}).
Suppose \eqref{f1}-\eqref{f5} and \eqref{sub} hold.
Assume that
\begin{equation}
\label{A2}
\mbox{$-\psi (x,t,z,p)$ and $A^{\xi \xi} (x,t,p)$ are concave
   in $p$}, \;\; \forall \, \xi \in T_x M,
\end{equation}
\begin{equation}
\label{A4}
\psi_z \leq 0.
\end{equation}
Then
\begin{equation}
\label{hess-a10}
\max_{\bM_T} |\nabla^2 u| \leq
 C_1 \big(1 + \max_{\mathcal{P} M_T}|\nabla^2 u|\big)
\end{equation}
where $C_1 > 0$ depends on $|u|_{C^1_x (\bM_T)}$ and $|\ul u|_{C^2 (\bM_T)}$.
Suppose that $u$ also satisfies the boundary condition
\eqref{eqn-bd} and, in addition,
assume that
\begin{equation}
\label{3I-45}
\sum f_{i} (\lambda) \lambda_{i} \geq 0,
     \; \forall \, \lambda \in \Gamma,
\end{equation}
\begin{equation}
\label{comp}
f (\lambda (\nabla^2 \varphi (x, 0) + A [\varphi (x, 0)])) - \varphi_t (x, 0)
   = \psi [\varphi (x, 0)],\ \ \forall x \in \bM,
\end{equation}
and
\begin{equation}
\label{lbd0}
\varphi_t (x, t) + \psi (x, t, z, p) > 0
\end{equation}
for each $(x, t) \in SM_T$, $p \in T^*_x \bM$ and $z \in \mathbb{R}$.
Then there exists $C_2 > 0$ depending on
$|u|_{C^1_x (\bM_T)}$, $|\ul u|_{C^2 (\bM_T)}$ and
$|\varphi|_{C^4 (\mathcal{P} M_T)}$ such that
\begin{equation}
\label{hess-a10b}
\max_{\mathcal{P} M_T}|\nabla^2 u| \leq C_2.
\end{equation}
\end{theorem}

Since $u$ is admissible, we have, by \eqref{A2},
\[
\triangle u + \mathrm{tr} A_{p_k} (x, t, 0) \nabla_k u
  + \mathrm{tr} A (x, t, 0) \geq \triangle u + \mathrm{tr} A (x, t, \nabla u) > 0
\]
and by the maximum principle it is easy to derive
the estimate
\begin{equation}
\label{gj-I115}
 \max_{\bM_T} |u| + \max_{\mathcal{P} M_T} |\nabla u| \leq C.
\end{equation}

Combining with the gradient estimates (Theorem \ref{p-th0}-\ref{jsui-th2}), we can prove
the following theorem immediately.
\begin{theorem}
\label{thm-main}
Let $u \in C^4 (\bM_T)$ be an
admissible solution of (\ref{eqn}) in $M_T$ with
$u \geq \ul u$ in $M_T$ and $u = \varphi$ on $\mathcal{P} M_T$.
Suppose \eqref{f1}-\eqref{f5}, \eqref{sub}-\eqref{A4},
and \eqref{3I-45}-\eqref{lbd0} hold. Then we have
\begin{equation}
\label{gsui-3}
|u|_{C^{2,1}_{x, t} (\bM_T)}\leq C,
\end{equation}
where $C > 0$ depends on $n$, $M$ and $|\ul u|_{C^2(\bM_T)}$
under any of the following additional assumptions:
(\romannumeral1) \eqref{A1-parabolic}-\eqref{A3}
hold for $\gamma_1 < 4$, $\gamma_2 = 2$
in \eqref{A1-parabolic};
(\romannumeral2) $(M^n, g)$ has nonnegative sectional curvature and
\eqref{A1-parabolic} hold
for $\gamma_1, \gamma_2 < 2$;
(\romannumeral3) \eqref{A1-parabolic}, \eqref{f4}-\eqref{p-G20**}
hold for $\gamma_1, \gamma_2 < 4$ in \eqref{A1-parabolic} and
$\gamma < 2$ in \eqref{p-A5}-\eqref{p-G20**}.
\end{theorem}

The rest of this paper is organized as follows. In Section 2, we introduce
some preliminaries and present a brief review
of some elementary formulas. In Section 3 and Section 4, we
establish the global and boundary estimates for second order
derivatives respectively. The gradient estimates are derived in
Section 5.

\section{Preliminaries}
\label{gj-P}
\setcounter{equation}{0}
\medskip

Throughout the paper $\nabla$ denotes the Levi-Civita connection
of $(M^n, g)$. The curvature tensor is defined by
\[ R (X, Y) Z = - \nabla_X \nabla_Y Z + \nabla_Y \nabla_X Z
                 + \nabla_{[X, Y]} Z. \]

Let $e_1, \ldots, e_n$ be local frames on $M^n$. We denote
$g_{ij} = g (e_i, e_j)$, $\{g^{ij}\} = \{g_{ij}\}^{-1}$.
Define the Christoffel symbols $\Gamma_{ij}^k$ by
$\nabla_{e_i} e_j = \Gamma_{ij}^k e_k$ and the curvature
coefficients
\[  R_{ijkl} = g( R (e_k, e_l) e_j, e_i), \;\; R^i_{jkl} = g^{im} R_{mjkl}.  \]
We shall use the notation $\nabla_i = \nabla_{e_i}$,
$\nabla_{ij} = \nabla_i \nabla_j - \Gamma_{ij}^k \nabla_k $, etc.

For a differentiable function $v$ defined on $M^n$, we usually
identity $\nabla v$ with the gradient of $v$, and use
$\nabla^2 v$ to denote the Hessian of $v$ which is locally given by
$\nabla_{ij} v = \nabla_i (\nabla_j v) - \Gamma_{ij}^k \nabla_k v$.
We recall that $\nabla_{ij} v =\nabla_{ji} v$ and
\begin{equation}
\label{hess-A70}
 \nabla_{ijk} v - \nabla_{jik} v = R^l_{kij} \nabla_l v,
\end{equation}
\begin{equation}
\label{hess-A80}
\begin{aligned}
\nabla_{ijkl} v - \nabla_{klij} v
= R^m_{ljk} \nabla_{im} v & + \nabla_i R^m_{ljk} \nabla_m v
      + R^m_{lik} \nabla_{jm} v \\
  & + R^m_{jik} \nabla_{lm} v
      + R^m_{jil} \nabla_{km} v + \nabla_k R^m_{jil} \nabla_m v.
\end{aligned}
\end{equation}

Let $u \in C^4 (\bM_T)$ be an admissible solution of
equation~\eqref{eqn}.
For simplicity we shall denote $U := \nabla^2 u + A (x, t, \nabla u)$
and, under a local frame $e_1, \ldots, e_n$,
\[ U_{ij} \equiv U (e_i, e_j) = \nabla_{ij} u + A^{ij} (x, t, \nabla u), \]
\begin{equation}
\label{gj-P10}
\begin{aligned}
\nabla_k U_{ij}
  \equiv \,& \nabla U (e_i, e_j, e_k)
             = \nabla_{kij} u + \nabla_k A^{ij} (x, t, \nabla u)  \\
  \equiv \,& \nabla_{kij} u  + A_{x_k}^{ij} (x, t, \nabla u)
              + A^{ij}_{p_l} (x, t, \nabla u) \nabla_{kl} u,
 \end{aligned}
\end{equation}
\begin{equation}
\label{gj-P10-1}
\begin{aligned}
(U_{ij})_t
  \equiv \,& (U (e_i, e_j))_t
             = (\nabla_{ij} u)_t + A^{ij}_t (x, t, \nabla u)
               + A^{ij}_{p_l} (x, t, \nabla u) (\nabla_{l} u)_t \\
  \equiv \,& \nabla_{ij} u_t  + A_{t}^{ij} (x, t, \nabla u)
              + A^{ij}_{p_l} (x, t, \nabla u) \nabla_{l} u_t,
 \end{aligned}
\end{equation}
 where $A^{ij} = A^{e_{i} e_{j}}$ and
$A_{x_k}^{ij}$ denotes the {\em partial} covariant derivative of $A$
when viewed as depending on $x \in M$ only, while the meanings of
$A^{ij}_t$ and $A^{ij}_{p_l}$, etc are obvious. Similarly we can calculate
 $\nabla_{kl} U_{ij} = \nabla_k \nabla_l U_{ij} - \Gamma_{kl}^m \nabla_m U_{ij}$, etc.

Let $F$ be the function defined by
\[ F (h) = f (\lambda (h)) \]
for a $(0,2)$ tensor $h$ on $M$.

Following the literature we denote throughout this paper
\[ F^{ij} = \frac{\partial F}{\partial h_{ij}} (U), \;\;
  F^{ij, kl} = \frac{\partial^2 F}{\partial h_{ij} \partial h_{kl}} (U) \]
under an orthonormal local frame $e_1, \ldots, e_n$.
The matrix $\{F^{ij}\}$ has eigenvalues $f_1, \ldots, f_n$ and
is positive definite by assumption (\ref{f1}), while (\ref{f2})
implies that $F$ is a concave function of $U_{ij}$ (see \cite{CNS}).
Moreover, when $\{U_{ij}\}$ is diagonal so is $\{F^{ij}\}$, and the
following identities hold
\[   F^{ij} U_{ij} = \sum f_i \lambda_i, \;\; F^{ij} U_{ik} U_{kj} = \sum f_i \lambda_i^2,
  \;\; \lambda (U) = (\lambda_1, \ldots, \lambda_n). \]

Define the linear operator $\mathcal{L}$ locally by
\[\mathcal{L} v = F^{ij} \nabla_{ij} v + (F^{ij} A^{ij}_{p_k} - \psi_{p_k}) \nabla_k v - v_t,\]
for $v \in C^{2, 1}_{x, t} (M_T)$.
We can prove
\begin{theorem}
\label{barrier}
Let $u$ be an admissible solution to \eqref{eqn} with $u \geq \ul u$ in $M_T$.
Assume that \eqref{f1}, \eqref{f2}, \eqref{A2} and \eqref{A4} hold. Then there exists a constant
$\theta > 0$ depending only on $\delta_0$ and $\ul u$ such that
\begin{equation}
\label{sub-bar}
\mathcal{L} (\ul u - u) \geq \theta (1 + \sum F^{ii})
\end{equation}
\end{theorem}
\begin{proof}
Since $\ul u$ is admissible satisfying \eqref{sub}, there exists a constant $\varepsilon_0 > 0$ such that
$\{x \in \bM_T: \lambda (\nabla^2 \ul u + A [\ul u] - \varepsilon_0 g)\}$ is a compact subset of $\Gamma$
and
\[
f (\lambda (\nabla^2 \ul u + A [\ul u] - \varepsilon_0 g))
- \ul u_t \geq \psi [\ul u] + \frac{\delta_0}{2} \mbox{ in $M_T$.}
\]
Let $\theta = \min \{\frac{\delta_0}{2}, \varepsilon_0\}$. For each $(x, t) \in M_T$, we may
assume $\{U_{ij}\} = \{\nabla_{ij} u + A^{ij}\}$ is diagonal at
$(x, t)$. From \eqref{A2}, \eqref{A4} and the concavity of $F$, we see, at $(x, t)$,
\begin{equation}
\label{para-100}
\begin{aligned}
F^{ii} (\ul U_{ii} - \varepsilon_0 g_{ii} - U_{ii}) - (\ul u - u)_t
   \geq \,& \psi (x, t, \ul u, \nabla \ul u) - \psi (x, t, u, \nabla u) + \frac{\delta_0}{2}\\
   \geq \,& \psi (x, t, u, \nabla \ul u) - \psi (x, t, u, \nabla u) + \frac{\delta_0}{2}\\
   \geq \,& \psi_{p_k} \nabla_k (\ul u - u) + \frac{\delta_0}{2}.
\end{aligned}
\end{equation}
By \eqref{A2} again, we have
\begin{equation}
\label{para-101}
\begin{aligned}
F^{ii} (\ul U_{ii} - U_{ii}) = \,& F^{ii} \nabla_{ii} (\ul u - u)
  + F^{ii} (A^{ii} (x, t, \nabla \ul u) - A^{ii} (x, t, \nabla u))\\
    \geq \,& F^{ii} \nabla_{ii} (\ul u - u) + F^{ii} A^{ii}_{p_k} \nabla_k (\ul u - u).
\end{aligned}
\end{equation}
Combining \eqref{para-100} and \eqref{para-101}, we get
\[
\mathcal{L} (\ul u - u) \geq \varepsilon_0 \sum F^{ii} + \frac{\delta_0}{2}
   \geq \theta (1 + \sum F^{ii})
\]
\end{proof}

\section{Global estimates for second derivatives}
\setcounter{equation}{0}

In this section, we prove \eqref{hess-a10} in Theorem \ref{jsui-th1} for which we set
\[W = \max_{(x,t) \in \bar{M_T}} \max_{\xi \in T_x M, |\xi| = 1}
     (\nabla_{\xi\xi} u + A^{\xi \xi} (x, u, \nabla u) e^\phi,\]
as in \cite{GJ}, where $\phi$ is a function to be determined. It suffices to estimate $W$.
We may assume $W$ is achieved at $(x_{0}, t_{0}) \in \bM_T - \mathcal{P} M_T$.
Choose a smooth orthonormal local frame $e_{1}, \ldots, e_{n}$ about $x_{0}$
such that $\nabla_i e_j = 0$,  and
$U$ is diagonal at $(x_0, t_0)$.
We assume $U_{11} (x_0, t_0) \geq \ldots \geq U_{nn} (x_0, t_0)$.
We have $W = U_{11} (x_0, t_0) e^{\phi (x_0, t_0)}$.

At the point $(x_{0}, t_{0})$ where the function
$\log U_{11} + \phi$ attains its maximum,
we have
\begin{equation}
\label{gs3}
\frac{\nabla_i U_{11}}{U_{11}} + \nabla_i \phi = 0
    \mbox{ for each } i = 1, \ldots, n,
\end{equation}
\begin{equation}
\label{gs4}
\frac{(U_{11})_t}{U_{11}} + \phi_t \geq 0,
\end{equation}
and
\begin{equation}
\label{gs5}
0 \geq \sum_{i} F^{ii} \{\frac{\nabla_{ii} U_{11}}{U_{11}}
   - \frac{(\nabla_i U_{11})^{2}}{U^{2}_{11}} + \nabla_{ii} \phi\}.
\end{equation}
Differentiating equation (\ref{eqn}) twice, we find
\begin{equation}
\label{gs1}
F^{ii} \nabla_{k} U_{ii} - \nabla_k u_t = \psi_{x_k} + \psi_{u} \nabla_{k} u
            + \psi_{p_j} \nabla_{kj} u, \mbox{ for all } k,
\end{equation}
and
\begin{equation}
\label{gs2}
\begin{aligned}
F^{ii} \nabla_{11} U_{ii}
     + \,& F^{ij,kl} \nabla_{1} U_{ij} \nabla_{1} U_{kl} - \nabla_{11} u_t \\
  \geq \,& \psi_{p_{j}} \nabla_{11j} u + \psi_{p_{l} p_{k}} \nabla_{1k} u \nabla_{1l} u
       - C U_{11} \\
  \geq \,& \psi_{p_{j}} \nabla_j U_{11} + \psi_{p_{1} p_{1}} U_{11}^2 - C U_{11} \\
    = \,& - U_{11} \psi_{p_{j}} \nabla_j \phi + \psi_{p_{1} p_{1}} U_{11}^2 - C U_{11}.
\end{aligned}
\end{equation}
Next, by \eqref{gs3} and \eqref{gs1},
\begin{equation}
\label{gj-S80}
\begin{aligned}
F^{ii}(\nabla_{ii} A^{11} - \nabla_{11} A^{ii})
 \geq \,& F^{ii} (A^{11}_{p_{j}} \nabla_{iij} u - A^{ii}_{p_{j}} \nabla_{11j} u) \\
      \,& + F^{ii} (A^{11}_{p_i p_i} U_{ii}^2 - A^{ii}_{p_1 p_1} U_{11}^2)
          - C U_{11} \sum F^{ii} \\
 \geq \,& U_{11} F^{ii} A^{ii}_{p_{j}} \nabla_{j} \phi
          + A^{11}_{p_j} \nabla_j u_t - C U_{11} \sum F^{ii} - C U_{11} \\
        & - C \sum_{i \geq 2} F^{ii} U_{ii}^2
          - U_{11}^2 \sum_{i \geq 2} F^{ii} A^{ii}_{p_1 p_1}.
\end{aligned}
\end{equation}
Note that
\begin{equation}
\label{gj-S50}
\nabla_{ii} U_{11}
  \geq \nabla_{11} U_{ii} + \nabla_{ii} A^{11} - \nabla_{11} A^{ii} - C U_{11}.
\end{equation}
Thus, by (\ref{gs2}), \eqref{gj-S80} and \eqref{gs4}, we have, at $(x_{0}, t_{0})$,
\begin{equation}
\label{gs6}
\begin{aligned}
F^{ii} \nabla_{ii} U_{11} \geq \,& F^{ii} \nabla_{11} U_{ii} - C U_{11} (1 + \sum F^{ii})
      + A^{11}_{p_j} \nabla_j u_t\\
    & - C \sum_{i \geq 2} F^{ii} U_{ii}^2
          - U_{11}^2 \sum_{i \geq 2} F^{ii} A^{ii}_{p_1 p_1} + U_{11} F^{ii} A^{ii}_{p_{j}} \nabla_{j} \phi\\
             \geq \,& U_{11} \mathcal{L} \phi - U_{11} F^{ii} \nabla_{ii} \phi
                - F^{ij,kl} \nabla_{1} U_{ij} \nabla_{1} U_{kl} + \psi_{p_{1} p_{1}} U_{11}^2\\
                  & - C U_{11} (1 + \sum F^{ii}) - C F^{ii} U_{ii}^2
                    - U_{11}^2 \sum_{i \geq 2} F^{ii} A^{ii}_{p_1 p_1}.
\end{aligned}
\end{equation}
It follows that, by \eqref{gs5},
\begin{equation}
\label{gs11}
\mathcal{L} \phi \leq U_{11} \sum_{i \geq 2} F^{ii} A^{ii}_{p_1 p_1} - \psi_{p_{1} p_{1}} U_{11}
    + C (1 + \sum F^{ii}) + \frac{C}{U_{11}} F^{ii} U^2_{ii} + E,
\end{equation}
where
\[E = \frac{1}{U^2_{11}} F^{ii} (\nabla_i U_{11})^{2}
    + \frac{1}{U_{11}} F^{ij,kl} \nabla_{1} U_{ij} \nabla_{1} U_{kl}.\]
Let
\[ \phi = \frac{\delta |\nabla u|^2}{2} + b \eta, \]
where $b$, $\delta$ are undetermined constants, $0 < \delta < 1 \leq b$, and $\eta$ is a $C^{2}$
function which may depend on $u$ but not on its derivatives.
We calculate, at $(x_0, t_0)$,
\begin{equation}
\label{ps-gs3}
\nabla_{i} \phi
     = \delta \nabla_{j} u \nabla_{ij} u + b \nabla_{i} \eta
     = \delta \nabla_i u U_{ii} - \delta \nabla_{j} u A^{ij} + b \nabla_{i} \eta
\end{equation}
\begin{equation}
\label{ps-gs301}
\phi_t
     = \delta \nabla_{j} u (\nabla_{j} u)_t + b \eta_t
\end{equation}
\begin{equation}
\label{ps-gs4}
\nabla_{ii} \phi
   \geq \frac{\delta}{2} U_{ii}^2 - C \delta + \delta \nabla_{j} u \nabla_{iij} u
            + b \nabla_{ii} \eta.
\end{equation}
From \eqref{hess-A70} and \eqref{gs1}, we derive
\begin{equation}
\label{ps-gs7}
\begin{aligned}
F^{ii} \nabla_{j} u \nabla_{iij} u
  \geq \,& F^{ii} \nabla_{j} u (\nabla_{j} U_{ii} - \nabla_{j} A^{ii})
            - C |\nabla u|^2 \sum F^{ii} \\
  \geq \,& (\psi_{p_{l}} - F^{ii} A^{ii}_{p_l}) \nabla_{j} u \nabla_{jl} u
           + \nabla_j u \nabla_j (u_t) - C (1 + \sum F^{ii}).
\end{aligned}
\end{equation}
Therefore,
\begin{equation}
\label{ps-S100}
\begin{aligned}
\mathcal{L} \phi \geq b \mathcal{L} \eta + \frac{\delta}{2} F^{ii} U_{ii}^2
   - C \sum F^{ii} - C.
\end{aligned}
\end{equation}
Let  $\eta = \ul u - u$. We get
from \eqref{ps-gs3} that
\begin{equation}
\label{bs-gs3.5}
\begin{aligned}
(\nabla_{i} \phi)^2
  \leq C \delta^2 (1 + U_{ii}^2) + 2 b^2 (\nabla_{i} \eta)^2
  \leq C \delta^2 U_{ii}^2 + C b^2.
\end{aligned}
\end{equation}
For fixed $0 < s \leq 1/3$ let
\[  \begin{aligned}
J \,& = \{i: U_{ii} \leq - s U_{11}\}, \;\;
K = \{i:  U_{ii} > - s U_{11} \}.
  \end{aligned} \]

Using a result of Andrews \cite{A} and Gerhardt \cite{GC}
as in \cite{G} and \cite{GJ} (see \cite{U} also), we have
\begin{equation}
\label{gj-S140}
 E \leq C b^2 \sum_{i \in J} F^{ii}  + C \delta^2 \sum F^{ii} U_{ii}^2
            +  C \sum F^{ii} + C (\delta^2 U_{11}^2 + b^2) F^{11}.
\end{equation}
Therefore, by \eqref{gs11}, \eqref{ps-S100} and \eqref{gj-S140}, we have
\begin{equation}
\label{ps-S150}
\begin{aligned}
b \mathcal{L} \eta
   \leq \,& \Big(C \delta^2 - \frac{\delta}{2} + \frac{C}{U_{11}}\Big) F^{ii} U_{ii}^2
            + C b^2 \sum_{i \in J} F^{ii} + C \sum F^{ii} \\
          & + C (\delta^2 U_{11}^2 + b^2) F^{11} + C\\
             \leq \,& \Big(C \delta^2 - \frac{\delta}{2} + \frac{C}{U_{11}}\Big) F^{ii} U_{ii}^2
            + C b^2 \sum_{i \in J} F^{ii} + C \sum F^{ii} \\
          & + C b^2 F^{11} + C.
\end{aligned}
\end{equation}
Choose $\delta$ sufficiently small such that $C \delta^2 - \frac{\delta}{2}$
is negative and let
\[
c_1 := - \frac{1}{2} \Big(C \delta^2 - \frac{\delta}{2}\Big) > 0.
\]
We may assume
\[
C \delta^2 - \frac{\delta}{2} + \frac{C}{U_{11}} \leq - c_1
\]
for otherwise we have $U_{11} \leq \frac{C}{c_1}$ and we are done. Thus, by \eqref{sub-bar},
choosing $b$ sufficiently large, we derive from \eqref{ps-S150} that
\[
c_1 F^{ii}U_{ii}^2 - C b^2 F^{11} - C b^2\sum_{i \in J}F^{ii} \leq 0.
\]
Then we can get a bound $U_{11} (x_0, t_0) \leq C$ since
$|U_{ii}| \geq s U_{11}$ for $i \in J$.
The proof of \eqref{hess-a10} is completed.

\section{Boundary estimates for second derivatives}
\setcounter{equation}{0}

In this section, we consider the estimates of second order derivatives on parabolic boundary
$\mathcal{P} M_T$. We may assume $\varphi \in C^4 (\bM_T)$.

Fix a point $(x_{0}, t_{0}) \in S M_T$. We shall choose
smooth orthonormal local frames $e_1, \ldots, e_n$ around $x_0$ such that
when restricted to $\partial M$, $e_n$ is normal to $\partial M$.
Since $u - \ul{u} = 0$ on $S M_T$ we have
\begin{equation}
\label{hess-a200}
\nabla_{\alpha \beta} (u - \ul{u})
 = -  \nabla_n (u - \ul{u}) \varPi (e_{\alpha}, e_{\beta}), \;\;
\forall \; 1 \leq \alpha, \beta < n \;\;
\mbox{on  $S M_T$},
\end{equation}
where 
$\varPi$ denotes the second fundamental form of $\partial M$.
Therefore,
\begin{equation}
\label{hess-E130}
|\nabla_{\alpha \beta} u| \leq  C,  \;\; \forall \; 1 \leq \alpha, \beta < n
\;\;\mbox{on} \;\; S M_T.
\end{equation}

Let $\rho (x)$ denote the distance from $x \in M$ to $x_{0}$,
\[\rho (x) \equiv \mathrm{dist}_{M^n} (x, x_{0}),\]
and set
\[M_{\delta} = \{X = (x, t) \in M \times (0,T]:
      \rho (x) < \delta, t \leq t_{0} + \delta\}.\]

For the mixed tangential-normal and pure normal second derivatives
at $(x_0, t_0)$, we shall use  the following barrier function as in
\cite{G},
\begin{equation}
\label{eq3-3}
 \varPsi
   = A_1 v + A_2 \rho^2 - A_3 \sum_{l< n} |\nabla_l (u - \varphi)|^2
\end{equation}
where
$v = u - \ul{u}$.
By differentiating the equation \eqref{eqn} and
straightforward calculation, we obtain
\begin{equation}
\label{eq3-1}
\begin{aligned}
\mathcal{L} (\nabla_k (u - \varphi))  \leq \,& C \Big(1 + \sum f_i |\lambda_i| + \sum
f_i\Big),
 \;\; \forall \; 1 \leq k \leq n.
\end{aligned}
\end{equation}
Similar to \cite{G} (see \cite{GJ} also), using Proposition 2.19 and Corollary 2.21 of
\cite{G} and Theorem \ref{barrier}, we can prove that there exist uniform positive constants
$\delta$ sufficiently small, and $A_1$, $A_2$, $A_3$
sufficiently large such that
\begin{equation}
\label{pbs-E170'}
\mathcal{L} (\varPsi \pm \nabla_{\alpha} (u - \varphi)) \leq 0
  \mbox{ in } M_\delta
\end{equation}
and
$\varPsi \pm \nabla_{\alpha} (u - \varphi) \geq 0$ on $\mathcal{P} M_{\delta}$.
Thus, by the maximum principle, we see $\varPsi \pm \nabla_{\alpha} (u - \varphi) \geq 0$
in $M_{\delta}$. Then we get
\begin{equation}
\label{pbs-E130'}
|\nabla_{n\alpha} u (x_0, t_0)| \leq \nabla_n \varPsi (x_0, t_0) \leq C,
\;\; \forall \; \alpha < n.
\end{equation}

It remains to derive
\begin{equation}
\label{cma-200}
\nabla_{nn} u (x_{0}, t_{0}) \leq C
\end{equation}
since $\triangle u \geq - C$.
We shall use an idea of Trudinger \cite{T} as \cite{G} and \cite{GJ} to prove that
there exist uniform positive constants $c_{0}$, $R_{0}$ such that for
all $R > R_{0}$,
$(\lambda' [U], R) \in \Gamma$  and
\begin{equation}
\label{pbs-bs9}
f(\lambda' [U], R) - u_{t} \geq \psi [u] + c_{0} \mbox{ on } \overline{S M_T}
\end{equation}
which implies \eqref{cma-200} by Lemma 1.2 in \cite{CNS}, where
$\lambda' [U] = (\lambda'_{1}, \ldots, \lambda'_{n-1})$
denote the eigenvalues of the $(n-1) \times (n-1)$ matrix
$\{U_{\alpha \beta}\}_{1 \leq \alpha, \beta \leq (n-1)}$
and $\psi [u] = \psi (\cdot, \cdot, u, \nabla u)$.
For $R > 0$ and a symmetric $(n-1)^2$
matrix $\{r_{\alpha  {\beta}}\}$ with $(\lambda' (\{r_{\alpha \beta}\}), R) \in \Gamma$ , define
\[G [r_{\alpha \beta}] \equiv f (\lambda' [\{r_{\alpha \beta}\}], R)\]
and consider
\[m_R \equiv \min_{(x, t) \in \overline{S M_T}}
     G [U_{\alpha \beta} (x, t)] - u_{t}(x, t) - \psi [u].\]
Note that  $G$ is concave and $m_R$ is increasing in $R$ by \eqref{f1},
and that
\[\begin{aligned}
c_R \equiv \,& \inf_{\overline{S M_T}} (G[\ul U_{\alpha \beta}] - \ul u_t - \psi [\ul u])\\
       \geq \,& \inf_{\overline{S M_T}} (G[\ul U_{\alpha \beta}] - F [\underline{U}_{ij}]) > 0
\end{aligned}\]
when $R$ is sufficiently large.

We wish to show $m_R > 0$ for $R$ sufficiently large. Without loss of generality
we assume $m_R < c_R/2$ (otherwise we are done) and
suppose $m_R$ is achieved at a point $(x_{0}, t_{0}) \in \overline{S M_T}$.
Choose local orthonormal frames around $x_0$ as before and assume
$\nabla_{n n} u (x_0, t_0) \geq \nabla_{n n} \ul u (x_0, t_0)$.
Let $\sigma_{\alpha {\beta}} = \langle \nabla_{\alpha} e_{\beta}, e_n \rangle$
and
\[G^{\alpha\beta}_{0} = \frac{\partial G}{\partial r_{\alpha\beta}}
    [U_{\alpha\beta} (x_{0},t_{0})].\]
Note that
$\sigma_{\alpha \beta} = \varPi (e_\alpha, e_\beta)$ on $\partial M$
and that
\begin{equation}
\label{pbs-c-200}
G^{\alpha {\beta}}_0 (r_{\alpha  {\beta}} - U_{\alpha {\beta}} (x_0, t_0))
    \geq G [r_{\alpha  {\beta}}] - G [U_{\alpha  {\beta}} (x_0, t_0)]
 \end{equation}
for any symmetric matrix $\{r_{\alpha  \beta}\}$ with $(\lambda' [\{r_{\alpha  \beta}\}], R) \in \Gamma$
by the concavity of $G$.

In particular, since $u_t = \ul u_t = \varphi_t$ on $\overline{SM_T}$, we have
\begin{equation}
\label{pbs-c-210}
\begin{aligned}
G^{\alpha {\beta}}_0 U_{\alpha  {\beta}} - \psi [u] - \varphi_t
\,& - G^{\alpha {\beta}}_0 U_{\alpha  {\beta}} (x_0, t_0) + \psi [u] (x_0, t_0) + u_t (x_0, t_0)\\
\geq \,& G [U_{\alpha  {\beta}}] - \psi [u] - u_t - m_R \geq 0
\end{aligned}
\end{equation}
on $\overline{SM_T}$.

From \eqref{hess-a200} we see that
\begin{equation}
\label{pbs-c-220}
 U_{\alpha {\beta}} = \ul{U}_{\alpha {\beta}}
    - \nabla_n (u - \ul{u}) \sigma_{\alpha {\beta}}
+ A^{\alpha \beta}[u] - A^{\alpha \beta}[\ul u] \mbox{ on $\ol{SM_T}$}.
\end{equation}

Note that at $(x_0, t_0)$, we have
\begin{equation}
\label{pbs-c-225}
\begin{aligned}
 \nabla_n (u - \ul{u}) G^{\alpha \beta}_0 \sigma_{\alpha {\beta}}
   = \,& G^{\alpha {\beta}}_0 (\ul{U}_{\alpha \beta} - U_{\alpha \beta})
          + G^{\alpha {\beta}}_0 (A^{\alpha \beta}[u] - A^{\alpha \beta}[\ul u]) \\
\geq \,& G[\ul{U}_{\alpha {\beta}}] - G[U_{\alpha {\beta}}]
          + G^{\alpha {\beta}}_0 (A^{\alpha \beta}[u] - A^{\alpha \beta}[\ul u]) \\
  =  \,& G[\ul{U}_{\alpha {\beta}}] - \psi [u] - u_t - m_R
          + G^{\alpha {\beta}}_0 (A^{\alpha \beta}[u] - A^{\alpha \beta}[\ul u]) \\
 \geq \,& c_R - m_R + \psi [\ul u] + \ul u_t - \psi [u] - u_t\\
        & + G^{\alpha {\beta}}_0 (A^{\alpha \beta}[u] - A^{\alpha \beta}[\ul u]) \\
 \geq \,& \frac{c_R}{2} + H [u] - H[\ul u]
\end{aligned}
\end{equation}
where $H [u] = G^{\alpha {\beta}}_0 A^{\alpha \beta} [u] - \psi [u]$.

Define
\[ \varPhi = - \eta \nabla_n (u - \ul u) + H [u] - \varphi_t + Q \]
where $\eta = G^{\alpha {\beta}}_0 \sigma_{\alpha {\beta}}$ and
\[ Q \equiv G^{\alpha {\beta}}_0 \nabla_{\alpha {\beta}} \ul u
          - G^{\alpha {\beta}}_0 U_{\alpha {\beta}} (x_0, t_0) + \psi [u] (x_0, t_0) + u_t (x_0, t_0). \]
By virtue of \eqref{pbs-c-210} and \eqref{pbs-c-220} we see that
$\varPhi \geq 0$ on $\overline{S M_T}$ and $\varPhi (x_0, t_0) = 0$.

Next, by \eqref{eq3-1} and \eqref{A2},
\[  \begin{aligned}
\mathcal{L} H \leq \,& H_z [u] \mathcal{L} u
   + H_{p_k} [u] \mathcal{L} \nabla_k u
   + F^{ij} H_{p_k p_l} [u] \nabla_{ki} u \nabla_{lj} u\\
   & + C (\sum F^{ii} + \sum f_i |\lambda_i| + 1) \\
   \leq \,& C (\sum F^{ii} + \sum f_i |\lambda_i| + 1) + H_z [u] \mathcal{L} u.
   \end{aligned} \]
Since $H_z [u] \geq 0$, by Theorem \ref{barrier}, we have
\[
\mathcal{L} u = \mathcal{L} (u - \ul u) + \mathcal{L} \ul u \leq C (1 + \sum F^{ii}).
\]
It follows that
\[
\mathcal{L} H \leq C (\sum F^{ii} + \sum f_i |\lambda_i| + 1).
\]
Therefore,
\begin{equation}
\label{pbs-gjB360}
\mathcal{L} \varPhi
  \leq C (\sum F^{ii} + \sum f_i |\lambda_i| + 1).
\end{equation}
By the compatibility condition\eqref{comp}, we find that
\[
c'_R \equiv \inf_{x \in \bM} G (\nabla_{\alpha \beta} \varphi + A [\varphi]) (x, 0)
    - \psi [\varphi] (x, 0) - \varphi_{t} (x, 0) > 0
\]
when $R$ is sufficiently large.
We may assume $m_R < \frac{c'_R}{2}$ (otherwise we are done).
For $x \in \bM$, by the concavity of $G$ again, we have
\[
\begin{aligned}
\varPhi (x, 0) = \,& G_0^{\alpha \beta} (U_{\alpha \beta} (x, 0) - U_{\alpha \beta} (x_0, t_0))\\
         & - \psi [u] (x, 0) - \varphi_t (x, 0) + \psi [u] (x_0, t_0) + u_t (x_0, t_0)\\
     = \,& G^{\alpha \beta}_0 (\nabla_{\alpha \beta} \varphi + A [\varphi] (x, 0) - U_{\alpha \beta} (x_0, t_0))\\
              & - \varphi_t (x, 0) + u_{t}(x_{0},t_{0}) + \psi [u](x_{0}, t_{0}) - \psi [\varphi] (x, 0)\\
     \geq \,& G (\nabla_{\alpha \beta} \varphi + A [\varphi]) (x, 0) - G (U_{\alpha\beta} (x_0, t_0))\\
         & - \varphi_t (x, 0) + u_{t}(x_{0},t_{0}) + \psi [u](x_{0}, t_{0}) - \psi [\varphi] (x, 0)\\
         \geq \,& c'_R - m_R > \frac{c'_R}{2}.
\end{aligned}
\]
It means that $\varPhi > 0$ on $B M_{T}$. Thus,
we get $\varPhi \geq 0$ on $\mathcal{P} M_{\delta}$.

Consider the function $\varPsi$ defined in \eqref{eq3-3} as before.
Similarly, there exist another group of constants $A_1 \gg A_2 \gg A_3 \gg 1$ such that
\begin{equation}
\label{pbs-cma-106}
  \left\{ \begin{aligned}
  & \mathcal{L} (\varPsi + \varPhi) \leq  0 \;\; \mbox{ in $M_{\delta}$},  \\
        & \varPsi + \varPhi \geq 0 \;\; \mbox{ on $\mathcal{P} M_{\delta}$}.
\end{aligned} \right.
\end{equation}
By the maximum principle we find $\varPsi + \varPhi \geq 0$ in $M_{\delta}$.
It follows that $\nabla_n \varPhi (x_0, t_0) \geq - \nabla_n \varPsi (x_0, t_0) \geq -C$.

Following \cite{GJ}, we write $u^s = s u + (1-s) \ul u$ and
\[ H [u^s] = G^{\alpha {\beta}}_0 A^{\alpha \beta} [u^s] - \psi [u^s].  \]
We have
\[  \begin{aligned}
H [u] - H[\ul u]
     = \,& \int_0^1 \frac{d H[u^s]}{dt} ds \\
     = \,& (u - \ul{u})  \int _0^1 H_z [u^s] ds
       + \sum \nabla_k (u - \ul{u}) \int _0^1 H_{p_k} [u^s] ds.
       \end{aligned}  \]
Therefore, at $(x_0, t_0)$,
\begin{equation}
\label{c-235}
 H [u] - H[\ul u]
     =  \nabla_n (u - \ul{u}) \int _0^1 H_{p_n} [u^s] ds
\end{equation}
and
\begin{equation}
\label{c-245}
\begin{aligned}
 \nabla_n H [u]
 = \,& \nabla_n H [\ul u]
        +  \sum \nabla_{kn} (u - \ul{u}) \int _0^1 H_{p_k} [u^s] ds \\
   \,& + \nabla_n (u-\ul{u}) \int _0^1 (H_z [u^s]
        + H_{x_{n} p_n} [u^s] + H_{z p_n} [u^s] \nabla_n u^s) ds \\
   \,& + \nabla_n (u-\ul{u}) \sum \int _0^1 H_{p_n p_l} [u^s] \nabla_{ln} u^s ds \\
  \leq \,&  \nabla_{nn} (u - \ul{u})
          \int _0^1 (H_{p_n} [u^s] + s H_{p_n p_n} [u^s] \nabla_n (u - \ul{u})) ds
           + C \\
   \leq \,&  \nabla_{nn} (u - \ul{u}) \int _0^1 H_{p_n} [u^s] ds + C
       \end{aligned}
 \end{equation}
since $H_{p_n p_n} \leq 0$, $\nabla_{nn} (u - \ul{u}) \geq 0$
and $\nabla_{n} (u - \ul{u}) \geq 0$.
It follows that
\begin{equation}
\label{c-255}
\begin{aligned}
 \nabla_n \varPhi (x_0, t_0)
   \leq \,& - \eta (x_0, t_0) \nabla_{nn} (x_0, t_0) +  \nabla_n H [u] (x_0, t_0) + C \\
   \leq \,& \Big(- \eta (x_0, t_0) + \int _0^1 H_{p_n} [u^s] (x_0, t_0) ds\Big)
               \nabla_{nn} u (x_0, t_0) + C.
       \end{aligned}
 \end{equation}
 By \eqref{pbs-c-225} and \eqref{c-235},
\begin{equation}
\label{c-230}
\eta (x_0, t_0) -  \int _0^1 H_{p_n} [u^s] (x_0, t_0) ds
                \geq \frac{c_R}{2 \nabla_n (u - \ul{u}) (x_0, t_0)}
                \geq \epsilon_1 c_R > 0
\end{equation}
for some uniform $\epsilon_1 > 0$ independent of $R$.
This gives
\begin{equation}
\label{cma-310}
\nabla_{nn} u (x_0, t_0) \leq  \frac{C}{\epsilon_1 c_R}.
\end{equation}

So we have an {\em a priori}
upper bound for all eigenvalues of $\{U_{ij} (x_0, t_0)\}$.
Now by \eqref{lbd0}, there exists a constant $\nu_0 > 0$ such that
\[
\inf_{(x, t) \in \ol{SM_T}} \varphi_t (x, t) + \psi (x, t, u, \nabla u) \geq \nu_0.
\]
It follows that
$\lambda [\{U_{ij} (x_0, t_0)\}]$ is contained in a compact
subset of $\Gamma$ by \eqref{f5},
and therefore
\[ m_R = G [U_{\alpha \beta} (x_0, t_0)] - u_{t}(x_0, t_0) - \psi [u] (x_0, t_0) > 0 \]
when $R$ is sufficiently large.
Then \eqref{pbs-bs9} is valid and the proof of \eqref{hess-a10b} is completed.

\section{Gradient estimates}
\setcounter{equation}{0}

In this section we establish the gradient estimates to prove Theorem~\ref{p-th0}-\ref{jsui-th2} below.
Throughout the section,
we assume \eqref{f1}-\eqref{f2}, \eqref{A2}
and the following growth conditions hold
\begin{equation}
\label{A1-parabolic}
\left\{ \begin{aligned}
   & p \cdot \nabla_x A^{\xi \xi} (x, t, z, p)
        \leq \bar{\psi}_1 (x, t, z) |\xi|^2 (1 + |p|^{\gamma_1}) \\
   & p \cdot \nabla_x \psi (x, t, z, p)  + |p|^2 \psi_z (x, t, z, p)
        \geq - \bar{\psi}_2 (x, t, z) (1 + |p|^{\gamma_2})
  \end{aligned} \right.
\end{equation}
for some functions $\bar{\psi}_1, \bar{\psi}_2 \geq 0$ and constants
$\gamma_1, \gamma_2  > 0$.

Since the proofs of Theorem~\ref{p-th0}-\ref{jsui-th2} are similar
to those of Theorem 6.1-6.3 in \cite{GJ}, we only provide a sketch here.
For more details we refer the reader to \cite{GJ} where the elliptic
Hessian equations are treated.
\begin{theorem}
\label{p-th0}
Let $u \in C^3 (\bM_T)$ be an
admissible solution of (\ref{eqn}).
Assume, in addition, that
\begin{equation}
\label{gj-I105}
 \lim_{\sigma \rightarrow \infty} f (\sigma {\bf 1}) = + \infty
\end{equation}
where ${\bf 1}= (1, \ldots, 1) \in \bfR^n$ and
there exists a constant $c_0 > 0$ such that
\begin{equation}
\label{A3}
A^{\xi \xi}_{p_k p_l} (x,t,p) \eta_k \eta_l \leq - c_0 |\xi|^{2} |\eta|^{2}
  + c_0 |g (\xi, \eta)|^2, \;
\forall \, \xi, \eta \in T_x M.
\end{equation}
Suppose that $\gamma_1 < 4$, $\gamma_2 = 2$
in \eqref{A1-parabolic},
and that there is an admissible function $\ul u \in C^2 (\bM_T)$.
Then
\begin{equation}
\label{3I-R60}
\max_{\bM_T} |\nabla u|
     \leq C_3 \big(1 + \max_{\mathcal{P} M_T} |\nabla u|\big)
\end{equation}
where $C_3$ is a positive constant depending on $|u|_{C^0 (\bM_T)}$ and $|\ul u|_{C^1_x (\bM_T)}$.
\end{theorem}
\begin{proof}
Let $w = |\nabla u|$ and $\phi$ a positive function to be determined.
Suppose the function $w \phi^{-a}$
achieves a positive maximum at an interior point $(x_0, t_0) \in M_T - \mathcal{P} M_T$
where $a < 1$ is a positive constant.
Choose a smooth orthonormal local frame $e_1, \ldots, e_n$
about $x_0$ such that $\nabla_{e_i} e_j = 0$ at $x_0$
and $\{U_{ij} (x_0, t_0)\}$ is diagonal.

The function $\log w - a \log \phi$ attains its maximum
at $(x_0, t_0)$ where for $i = 1, \ldots, n$,
\begin{equation}
\label{g1-p}
\frac{\nabla_i w}{w} - \frac{a \nabla_i \phi}{\phi} = 0,
\end{equation}
\begin{equation}
\label{g1-p'}
\frac{w_t}{w} - \frac{a \phi_t}{\phi} \geq 0
\end{equation}
and
\begin{equation}
\label{g2-p}
\frac{\nabla_{ii} w}{w} + \frac{(a - a^2) |\nabla_{i} \phi|^2}{\phi^2}
   - \frac{a \nabla_{ii} \phi}{\phi} \leq 0.
\end{equation}
Note that
\[ w \nabla_i w = \nabla_{l} u \nabla_{il} u, \ w w_t = \nabla_l u (\nabla_l u)_t. \]

By (\ref{hess-A70}), (\ref{g1-p}) and \eqref{gs1},
\begin{equation}
\label{g3-p}
\begin{aligned}
w \nabla_{ii} w
    = \,& \nabla_{l} u \nabla_{iil} u + \nabla_{il} u \nabla_{il} u
          - \nabla_i w \nabla_i w \\
    = \,& (\nabla_{lii}u+ R^{k}_{iil} \nabla_{k} u) \nabla_{l} u
          + \Big(\delta_{kl} - \frac{\nabla_{k} u \nabla_{l} u}{w^2} \Big)
            \nabla_{ik} u \nabla_{il} u \\
 \geq \,& (\nabla_l U_{ii} - A^{ii}_{p_k} \nabla_{lk} u - A^{ii}_{x_l}) \nabla_l u
         - C |\nabla u|^2\\
    = \,& \nabla_{l} u \nabla_l U_{ii} - \frac{a w^2}{\phi} A^{ii}_{p_{k}} \nabla_{k} \phi
           - \nabla_l u A^{ii}_{x_l} - C w^2.
\end{aligned}
\end{equation}
By \eqref{gs1}, \eqref{g1-p} and \eqref{g1-p'},
\begin{equation}
\label{p-g11'}
\begin{aligned}
 F^{ii} \nabla_{l} u \nabla_l U_{ii}
    = \,& \nabla_{l} u \psi_{x_l} + \psi_u |\nabla u|^2
          + \psi_{p_k} \nabla_{l} u \nabla_{lk} u + \nabla_l u \nabla_l u_t\\
    \geq \,& \nabla_{l} u \psi_{x_l} + \psi_u |\nabla u|^2
          + \frac{a w^2}{\phi} \psi_{p_k} \nabla_k \phi + \frac{a w^2}{\phi} \phi_t.
        \end{aligned}
\end{equation}
Let
$\phi = (u - \ul u) + b > 0$, where $b = 1 + \sup_{M_T} (\ul u - u)$.

By \eqref{A3} we have
\begin{equation}
\label{g3.5}
\begin{aligned}
  - A^{ii}_{p_{k}} \nabla_{k} \phi
      = \,&  A^{ii}_{p_{k}} (x, t, \nabla u) \nabla_{k} (\ul u - u) \\
   \geq \,& A^{ii} (x, t, \nabla \ul{u}) - A^{ii} (x, t, \nabla u)
             + \frac{c_0}{2} (|\nabla \phi|^2 - |\nabla_i \phi|^2).
\end{aligned}
\end{equation}

We may assume that $c_0$ is sufficiently small and that
\[\frac{2 a - 2 a^2 - c_0 a \phi}{2 \phi^2} > 0\]
by choosing $a$ sufficiently small.

Thus, by (\ref{g2-p}), \eqref{g3-p}, \eqref{p-g11'} and \eqref{g3.5}, we find
\begin{equation}
\label{g10-p}
\begin{aligned}
0 \geq \,& \frac{a}{\phi} F^{ii} (\ul U_{ii} - U_{ii})
     + \frac{a c_0 |\nabla \phi|^2}{2 \phi} \sum F^{ii}
          + \frac{2a - 2a^2 - c_0 a \phi}{2 \phi^2} F^{ii} |\nabla_{i} \phi|^2\\
        & - \frac{1}{w^2} F^{ii} A^{ii}_{x_l} \nabla_{l} u
          + \frac{1}{w^2} \psi_{x_l} \nabla_{l} u + \psi_u + \frac{a}{\phi} \psi_{p_k} \nabla_k \phi
           + \frac{a}{\phi} \phi_t - C \sum F^{ii}\\
  \geq \,& \frac{a}{\phi} F^{ii} (\ul U_{ii} - U_{ii})
     + \frac{a c_0 |\nabla \phi|^2}{2\phi} \sum F^{ii} - C \sum F^{ii}\\
       & + \frac{a}{\phi} (\psi (x, t, u, \nabla u) - \psi (x, t, u, \nabla \ul u))\\
        & - \frac{1}{w^2} F^{ii} A^{ii}_{x_l} \nabla_{l} u
          + \frac{1}{w^2} \psi_{x_l} \nabla_{l} u + \psi_u
           + \frac{a}{\phi} (u - \ul u)_t
\end{aligned}
\end{equation}

Choose $B > 0$ sufficiently large such that (see \cite{GJ})
\[ F (2 B g + \ul U) \geq F (B g)   \;\; \mbox{in $\bM_T$}. \]

Therefore, by the concavity of $F$,
\begin{equation}
\label{g12}
\begin{aligned}
F^{ii} (\ul{U}_{ii}- U_{ii}) 
 & \geq F (2 B g + \ul U) - F (U) - 2 B \sum F^{ii}\\
& \geq F (Bg) - 2 B \sum F^{ii} - \psi (x, t, u, \nabla u) - u_t.
\end{aligned}
\end{equation}

It follows from \eqref{A1-parabolic}, \eqref{gj-I105}, \eqref{g10-p} and \eqref{g12} that
\begin{equation}
\label{g10-p'}
\begin{aligned}
0 \geq \,& \frac{a}{\phi} F (Bg) - C - (C + 2B) \sum F^{ii}
     + \frac{a c_0 |\nabla \phi|^2}{2\phi} \sum F^{ii} \\
        & - \frac{1}{w^2} F^{ii} A^{ii}_{x_l} \nabla_{l} u
          + \frac{1}{w^2} \psi_{x_l} \nabla_{l} u + \psi_u\\
  \geq \,& (\frac{a c_0 |\nabla \phi|^2}{2\phi} - 3 B - C |\nabla u|^{\gamma_1 - 2})\sum F^{ii}
\end{aligned}
\end{equation}
provided $B$ is chosen sufficiently large. Thus, we get a bound $|\nabla u (x_0, t_0)| \leq C$
and so the proof of Theorem \ref{p-th0} is completed.
\end{proof}

\begin{theorem}
\label{p-th1a}
Let $u \in C^3 (\bM_T)$ be an
admissible solution of (\ref{eqn}) with $u \geq \ul u$ in $M_T$.
Assume, in addition, that \eqref{sub},
\eqref{A4} and \eqref{A1-parabolic} hold for $\gamma_1, \gamma_2 < 2$
in \eqref{A1-parabolic} and that
$(M^n, g)$ has nonnegative sectional curvature.
Then \eqref{3I-R60} holds.
\end{theorem}
\begin{proof}
Since $(M^n, g)$ has nonnegative sectional curvature, in orthonormal local
frame,
\[ R^{k}_{iil} \nabla_{k} u \nabla_{l} u \geq 0. \]
In the proof of Theorem \ref{p-th0}, similar to \eqref{g3-p}, we have
\begin{equation}
\label{g3-p-1}
\begin{aligned}
w \nabla_{ii} w
     \geq \nabla_{l} u \nabla_l U_{ii} - \frac{a w^2}{\phi} A^{ii}_{p_{k}} \nabla_{k} \phi
           - \nabla_l u A^{ii}_{x_l}.
\end{aligned}
\end{equation}
It follows from \eqref{sub-bar}, \eqref{A1-parabolic}, \eqref{g2-p}, \eqref{p-g11'} and \eqref{g3-p-1} that
\begin{equation}
\label{g10-p-1}
\begin{aligned}
0 \geq \,& \frac{a}{\phi} \mathcal{L} (\ul u - u) + \frac{1}{w^2} \nabla_{l} u \psi_{x_l}
           + \psi_u - \frac{\nabla_{l} u}{w^2} F^{ii} A^{ii}_{x_l}
           + \frac{a - a^2}{\phi^2} F^{ii} |\nabla_{i} \phi|^2\\
   \geq \,& \frac{a}{\phi} \theta (1 + \sum F^{ii}) - C |\nabla u|^{\gamma_1 - 2} \sum F^{ii}
     - C |\nabla u|^{\gamma_2 - 2} + \frac{a - a^2}{\phi^2} F^{ii} |\nabla_{i} \phi|^2
\end{aligned}
\end{equation}
provided $|\nabla u|$ is sufficiently large. Choosing $a$ sufficiently small, we can obtain
a bound $|\nabla u (x_0, t_0)| \leq C$ and \eqref{3I-R60} holds.
\end{proof}
\begin{theorem}
\label{jsui-th2}
Let $u \in C^3 (\bM_T)$ be an
admissible solution of (\ref{eqn}) in $M_T$. Assume, in addition, that
\eqref{A1-parabolic} hold for $\gamma_1, \gamma_2 < 4$,
\begin{equation}
\label{f4}
f \mbox{ is homogeneous of degree one, }
\end{equation}
\begin{equation}
\label{f6}
f_j (\lambda) \geq \nu_1 \Big(1 + \sum f_{i} (\lambda)\Big) \mbox{ for any }
  \lambda \in \Gamma \mbox{ with } \lambda_j < 0,
\end{equation}
where $\nu_1$ is a uniform positive constant
and there exist a continuous function $\bar{\psi} \geq 0$ and a positive constant
$\gamma < 2$ such that when $|p|$ is sufficiently large,
\begin{equation}
\label{p-A5}
p \cdot D_p \psi (x, t, z, p),
    \; - p \cdot D_p A^{\xi \xi} (x, t, z, p)/|\xi|^2
     \leq \bar{\psi}(x, t, z)  (1 + |p|^{\gamma}),
\end{equation}
\begin{equation}
\label{p-A500}
- \psi (x, t, z, p) \leq \bar{\psi}(x, t, z)  (1 + |p|^{\gamma}),
\end{equation}
\begin{equation}
\label{p-G20**}
|A^{\xi \eta} (x, t, z, p)|
     \leq \bar{\psi} (x, t, z) |\xi||\eta| (1 + |p|^{\gamma}),
\;\; \forall \, \xi, \eta \in T_x \bM; \xi \perp \eta.
\end{equation}
Then \eqref{3I-R60} holds.
\end{theorem}
\begin{proof}
In the proof of Theorem \ref{p-th0}, we take
$\phi = - u + \sup_{M_T} u + 1$.
By the concavity of $A^{ii}$ with respect to $p$,
\begin{equation}
\label{p-G20}
 A^{ii} = A^{ii} (x, t, \nabla u)
\leq A^{ii} (x, t, 0) + A^{ii}_{p_k} (x, t, 0) \nabla_k u
\end{equation}
Thus, from \eqref{f4}, \eqref{p-A500} and \eqref{p-G20}, we find
\begin{equation}
\label{p-g17}
\begin{aligned}
- F^{ii} \nabla_{ii} \phi = F^{ii} \nabla_{ii} u = F^{ii} U_{ii} - \,& F^{ii} A^{ii}
   = u_t + \psi - F^{ii} A^{ii}\\
     \geq \,& u_t + \psi - C (1 + |\nabla u|) \sum F^{ii}\\
     \geq \,& u_t - C (1 + |\nabla u|) \sum F^{ii} - C |\nabla u|^{\gamma}.
\end{aligned}
\end{equation}
By virtue of \eqref{g2-p}, \eqref{g3-p}, \eqref{p-g11'}, \eqref{A1-parabolic},
\eqref{p-A5} and \eqref{p-g17}, we see that for $a < 1$,
\begin{equation}
\label{g18-p}
\begin{aligned}
0 \geq & \frac{(a - a^2)}{\phi^2} F^{ii} |\nabla_{i} u|^2
       + \frac{\nabla_{l} u \psi_{x_l}}{w^2} + \psi_u
           - \frac{a}{\phi} \psi_{p_k} \nabla_k u - \frac{a}{\phi} u_t \\
        & + \frac{a}{\phi} F^{ii} A^{ii}_{p_{k}} \nabla_{k} u
            - F^{ii} \frac{\nabla_l u A^{ii}_{x_l}}{w^2}
            + \frac{a}{\phi} u_t\\
            & - C |\nabla u|^{\gamma} - C (1 + |\nabla u|)\sum F^{ii} \\
  \geq & c_1 F^{ii} |\nabla_{i} u|^2 - C (|\nabla u|^{\gamma_2 - 2} + |\nabla u|^\gamma)\\
             & - C(1 + |\nabla u| + |\nabla u|^{\gamma_1 - 2} + |\nabla u|^\gamma)  \sum F^{ii}
\end{aligned}
\end{equation}
provided $|\nabla u|$ is sufficiently large.

Without loss of generality we assume
$\nabla_{1} u (x_{0}, t_0) \geq \frac{1}{n} |\nabla u (x_{0}, t_0)| > 0$.
Recall that $U_{ij} (x_0, t_0)$ is diagonal. By \eqref{g1-p}, \eqref{p-G20} and
\eqref{p-G20**}, we have
\begin{equation}
\label{p-G30}
\begin{aligned}
U_{11} = \,& - \frac{a}{\phi} |\nabla u|^2 + A^{11}
 + \frac{1}{\nabla_1 u} \sum_{k \geq 2} \nabla_k u A^{1k}\\
 \leq \,& - \frac{a}{\phi} |\nabla u|^2
      + C (1 + |\nabla u| + |\nabla u|^{\gamma-2}) < 0
\end{aligned}
\end{equation}
provided $|\nabla u|$ is sufficiently large. Therefore, by \eqref{f4},
\[ f_{1} \geq \nu_{0} \Big(1 + \sum^n_{i = 1} f_{i}\Big) \]
and a bound $|\nabla u (x_0, t_0)| \leq C$ follows from \eqref{g18-p}.
\end{proof}

{\bf Acknowledgement.}
This is an improvement of part of my thesis. I wish to thank my adviser
Professor Bo Guan for leading me to this problem and many useful suggestions
and comments.

\end{document}